\documentclass[11pt,leqno]{article}
\usepackage{xcolor}
\definecolor{labelkey}{rgb}{0,0.08,0.45}
\definecolor{refkey}{rgb}{0,0.6,0.0}
\definecolor{Brown}{rgb}{0.45,0.0,0.05}
\usepackage{mathpazo}
\usepackage{exscale,relsize}
\usepackage{amsmath}
\usepackage{amsfonts}
\usepackage{amssymb}
\usepackage{calc}
\usepackage{theorem}
\usepackage{pifont}      
\usepackage{graphicx}
\newcommand{\HH}{\ensuremath{\mathcal H}}
\newcommand{\gr}{\ensuremath{\operatorname{gra}}}

\newcommand{\scal}[2]{\langle{{#1},{#2}}\rangle}

\newcommand{\NN}{\ensuremath{\mathbb N}}
\newcommand{\nnn}{\ensuremath{{n \in \NN}}}

\newcommand{\menge}[2]{\big\{{#1} \mid {#2}\big\}}

\newcommand{\To}{\ensuremath{\rightrightarrows}}

\newcommand{\Id}{\ensuremath{\operatorname{Id}}}

\newcommand{\weakly}{\ensuremath{\,\rightharpoonup}\,}

\renewcommand{\phi}{\ensuremath{\varphi}}

\newtheorem{theorem}{Theorem}
\newtheorem{lemma}[theorem]{Lemma}

\newtheorem{corollary}[theorem]{Corollary}

\theoremstyle{plain}{\theorembodyfont{\rmfamily}
}
\theoremstyle{plain}{\theorembodyfont{\rmfamily}
}
\theoremstyle{plain}{\theorembodyfont{\rmfamily}
}
\theoremstyle{plain}{\theorembodyfont{\rmfamily}
}
\theoremstyle{plain}{\theorembodyfont{\rmfamily}
\newtheorem{remark}[theorem]{Remark}}
\theoremstyle{plain}{\theorembodyfont{\rmfamily}
}

\begin{document}


\title{\textsc{A note on the paper by 
Eckstein and Svaiter on ``General projective splitting
methods for\\ sums of maximal monotone operators''}}

\author{
Heinz H.\ Bauschke\thanks{Mathematics, Irving K.\ Barber School,
University of British Columbia Okanagan, 
Kelowna, B.C.\ V1V 1V7, Canada. E-mail:
\texttt{heinz.bauschke@ubc.ca}.}
}
 \vskip 3mm

\date{May 20, 2009}
\maketitle

\begin{abstract} \noindent
In their recent \emph{SIAM J.\ Control Optim.}\ paper from 2009, 
J.\ Eckstein and B.F.\ Svaiter proposed a very general and flexible 
splitting framework for finding a zero of the sum of 
finitely many maximal monotone operators.
In this short note, we provide a technical result that allows for the
removal of Eckstein and Svaiter's assumption that 
the sum of the operators be maximal monotone or that 
the underlying Hilbert space be finite-dimensional. 
\end{abstract}



\noindent
Throughout, we assume that
$\HH$ is a real Hilbert space with 
inner product $\scal{\cdot}{\cdot}$ and induced norm $\|\cdot\|$.
We shall assume basic notation and results from
Fixed Point Theory and from Monotone Operator Theory;
see, e.g., \cite{BurIus,GoeKir,GoeRei,RockWets,Si,Si2,Zalinescu}. 
The \emph{graph} of a maximal monotone operator $A\colon\HH\To\HH$ is
denoted by $\gr A$, and its \emph{resolvent} $(A+\Id)^{-1}$ by $J_A$. 
Weak convergence is indicated by $\weakly$. 

\begin{lemma}
\label{l:090519:1} 
Let $C$ be a closed linear subspace of $\HH$ and let
$F\colon \HH\to\HH$ be firmly nonexpansive. 
Then $P_C F + (\Id-P_C)(\Id-F)$ is firmly nonexpansive.
\end{lemma}
\begin{proof}
Since $P_C$ and $F$ are firmly nonexpansive,
we have that $2P_C-\Id$ and $2F-\Id$ are both nonexpansive.
Set $T = P_C F + (\Id-P_C)(\Id-F)$.
Then $2T-\Id = (2P_C-\Id)(2F-\Id)$ is nonexpansive,
and hence $T$ is firmly nonexpansive.
\end{proof}

\begin{theorem}
\label{t:main}
Let $A\colon \HH\To\HH$ be maximal monotone, 
and let $C$ be a closed linear subspace of $\HH$.
Let $(x_n,u_n)_\nnn$ be a sequence in $\gr A$ such
that $(x_n,u_n)\weakly (x,u)\in \HH\times \HH$. 
Suppose that $x_n-P_Cx_n\to 0$ and that
$P_Cu_n\to 0$, where $P_C$ denotes the projector onto $C$. 
Then $(x,u)\in(\gr A)\cap(C\times C^\bot)$ and
$\scal{x_n}{u_n}\to \scal{x}{u}=0$. 
\end{theorem}
\begin{proof}
Since $P_C$ is a bounded linear operator, it is weakly continuous
(\cite[Theorem~VI.1.1]{Conway}). 
Thus 
$x \leftharpoonup x_n = (x_n-P_Cx_n) +P_Cx_n \weakly 0 + P_Cx$ 
and hence
$x=P_Cx\in C$.
Similarly, $0 \leftarrow P_Cu_n\weakly P_Cu$; hence $P_Cu=0$ and so
$u\in C^\bot$. Altogether, 
\begin{equation} \label{e:090211:c---}
(x,u)\in C\times C^\bot.
\end{equation}
Since $\Id-J_A$ is firmly nonexpansive, we see 
from Lemma~\ref{l:090519:1} that 
\begin{equation} \label{e:090211:c-}
T = P_C(\Id-J_A) + (\Id-P_{C})J_A = P_C + (\Id-2P_C)J_A
\end{equation}
is also firmly nonexpansive.
Now $(\forall\nnn)$ $u_n\in Ax_n$, i.e., 
\begin{equation} \label{e:090519:a}
(\forall\nnn)\quad x_n = J_A(x_n+u_n).
\end{equation}
Furthermore, 
\begin{equation} \label{e:090211:c}
x_n+u_n \weakly x+u,
\end{equation}
and \eqref{e:090211:c-} and \eqref{e:090519:a} imply that 
$T(x_n+u_n) = P_C(x_n+u_n)+(\Id-2P_C)J_A(x_n+u_n)
= P_Cx_n+P_Cu_n +(\Id-2P_C)x_n =x_n-P_Cx_n + P_Cu_n \to 0$, 
i.e., that 
\begin{equation} \label{e:090211:d}
T(x_n+u_n)\to 0.
\end{equation}
Since $\Id-T$ is (firmly) nonexpansive,
the demiclosedness principle (see \cite{GoeKir,GoeRei}), applied
to the sequence $(x_n+u_n)_\nnn$ and the operator $\Id-T$, and
\eqref{e:090211:c} and \eqref{e:090211:d} imply that 
$(\Id-(\Id-T))(x+u) = 0$, i.e., that $T(x+u)=0$. 
Using \eqref{e:090211:c-}, this means that
\begin{equation} \label{e:090211:e}
J_A(x+u) = 2P_CJ_A(x+u)-P_C(x+u) \in C.
\end{equation}
Applying $P_C$ to both sides of \eqref{e:090211:e}, we deduce that 
$J_A(x+u)=P_CJ_A(x+u)$; consequently, \eqref{e:090211:e} simplifies to
\begin{equation} \label{e:090211:f}
J_A(x+u) = P_Cx+P_Cu.
\end{equation}
However, \eqref{e:090211:c---} yields
$P_Cx=x$ and $P_Cu=0$, hence \eqref{e:090211:f} becomes
$J_A(x+u)=x$; equivalently, $u\in Ax$ or
\begin{equation}
\label{e:090519:b}
(x,u)\in\gr A.
\end{equation}
Combining \eqref{e:090211:c---} and \eqref{e:090519:b}, we
see that $(x,u)\in(\gr A)\cap (C\times C^\bot)$, as claimed.
Finally, 
$\scal{x_n}{u_n} = \scal{P_Cx_n}{P_Cu_n} +
\scal{P_{C^\bot}x_n}{P_{C^\bot}u_n} \to
\scal{P_Cx}{0}+\scal{0}{P_{C^\bot}u} = 0 
= \scal{P_Cx}{P_{C^\bot}u} = \scal{x}{u}$. 
\end{proof}

\begin{corollary} 
\label{c:main}
Let $A_1,\ldots,A_m$ be maximal monotone operators $\HH$,
and let $z_1,\ldots,z_m$ and $w_1,\ldots,w_m$ be vectors in $\HH$.
Suppose that for each $i$, 
$(x_{i,n},y_{i,n})_\nnn$ is a sequence in $\gr A_i$ such that 
for all $i$ and $j$, 
\begin{align}
\label{e:drea:1}(x_{i,n},y_{i,n})&\weakly (z_i,w_i)\\
\label{e:drea:2}\sum_{i=1}^{m} y_{i,n}&\to 0\\
\label{e:drea:3}x_{i,n}-x_{j,n}&\to 0.
\end{align}
Then $z_1=\cdots=z_n$, $w_1+\cdots+w_n=0$,
and each $w_i\in A_iz_i$. 
\end{corollary}
\begin{proof}
We work in product Hilbert space
$\boldsymbol{\HH} = \HH^m$, 
and we set
\begin{equation}
\mathbf{A} = A_1\times\cdots\times A_m, \;\;\text{and}\;\;
\mathbf{C} = \menge{(x_1,\ldots,x_m)\in\boldsymbol{\HH}}{x_1=\cdots=x_m}. 
\end{equation}
Note that $\mathbf{A}$ is maximal monotone on $\boldsymbol{\HH}$, and
that $\mathbf{C}$ is a closed linear subspace of $\boldsymbol{\HH}$. 
Next, set 
$\mathbf{x} = (z_1,\ldots,z_m)$, 
$\mathbf{u} = (w_1,\ldots,w_m)$, 
and 
$(\forall\nnn)$ $\mathbf{x}_n = (x_{1,n},\ldots,x_{m,n})$ and
$\mathbf{u}_n = (y_{1,n},\ldots,y_{m,n})$. 
By \eqref{e:drea:1}, $(\mathbf{x}_n,\mathbf{u_n})_\nnn$ is a sequence
in $\gr \mathbf{A}$ such that $(\mathbf{x}_n,\mathbf{u}_n)\weakly
(\mathbf{x},\mathbf{u})$. 
Furthermore, \eqref{e:drea:2} and \eqref{e:drea:3} imply that
$P_{\mathbf{C}}\mathbf{u}_n\to 0$ and that
$\mathbf{x}_n-P_{\mathbf{C}}\mathbf{x}_n\to 0$, respectively. 
Therefore, by Theorem~\ref{t:main}, $(\mathbf{x},\mathbf{u})
\in (\gr\mathbf{A})\cap(\mathbf{C}\times\mathbf{C}^\bot)$,
which is precisely the announced conclusion. 
\end{proof}

\begin{remark} 
Corollary~\ref{c:main} is a considerable strengthening of
\cite[Proposition~A.1]{EckSvai09}, where it was 
\emph{additionally assumed} that
$A_1+\cdots+A_m$ is maximal monotone, and where 
part of the \emph{conclusion} of Corollary~\ref{c:main}, namely 
$z_1=\cdots=z_m$, was an additional \emph{assumption}. 
\end{remark}

\begin{remark}
Because of the removal of the assumption that $A_1+\cdots+A_m$ be
maximal monotone (see the previous remark), 
a second look at the proofs in Eckstein and Svaiter's
paper \cite{EckSvai09} reveals that --- in our present notation ---
the assumption that 
\begin{center}``either $\HH$ is finite-dimensional or
$A_1+\cdots+A_m$ is maximal monotone'' 
\end{center} is superfluous in both
\cite[Proposition~3.2 and Proposition~4.2]{EckSvai09}. 
This is important in the infinite-dimensional case, where
the maximality of the sum can typically be only guaranteed when
a constraint qualification is satisfied; consequently,
Corollary~\ref{c:main} helps to widen the scope of the powerful
algorithmic framework of Eckstein and Svaiter. 
\end{remark}

\section*{Acknowledgment}
The author was partially supported by the Natural Sciences and
Engineering Research Council of Canada and
by the Canada Research Chair Program.


\end{document}